\documentclass[psamsfonts]{amsart}
\usepackage{amssymb,amsfonts}
\usepackage[all,arc]{xy}
\usepackage{enumerate}
\usepackage{mathrsfs}
\usepackage{ stmaryrd }
\usepackage{hyperref}

\newcounter{qcounter}

\define\bP{\mathbb{P}}

\define\isoto{\xrightarrow{\sim}}
\define\onto{\twoheadrightarrow}

\define\Spec{\mathrm{Spec}}

\newcommand{\ttmat}[4]{\left( \begin{array}{cc}
#1 & #2 \\
#3 & #4
\end{array}
\right)}

\newcommand{\Z}{\mathbb{Z}}
\newcommand{\Q}{\mathbb{Q}}

\newcommand{\F}{\mathbb{F}}

\newcommand{\sO}{\mathcal{O}}

\newcommand{\I}{\mathcal{I}}

\newcommand{\m}{\mathfrak{m}}

\newcommand{\Aut}{\mathrm{Aut}}

\newcommand{\End}{\mathrm{End}}

\newcommand{\X}{\mathfrak{X}}

\define\cA{\mathcal{A}}

\define\GL{{\mathrm{GL}}}

\define\cC{\mathcal{C}}
\define\cR{\mathcal{R}}
\define\A{\mathcal{A}}
\define\cB{\mathcal{B}}

\define\fI{\mathfrak{I}}
\define\full{{\mathrm{full}}}
\define\Prim{{\mathrm{Prim}}}

\define\Isom{{\mathrm{\underline{Isom}}}}
\define\uHom{{\mathrm{\underline{Hom}}}}

\define\Ann{{\mathrm{Ann}}}
\define\KMD{{\mathrm{KM+D}}}
\define\gr{{\mathrm{gr}}}

\define\Zp{{\Z_{(p)}}}
\define\Tr{\mathrm{Tr}}
\newtheorem{thm}{Theorem}[section]

\newtheorem*{thmB}{Theorem \ref{thm: KMD}}
\newtheorem*{keylem}{Key Lemma}
\newtheorem{cor}[thm]{Corollary}
\newtheorem{prop}[thm]{Proposition}
\newtheorem{lem}[thm]{Lemma}

\theoremstyle{definition}
\newtheorem{defn}[thm]{Definition}

\newtheorem{exmp}[thm]{Example}

\theoremstyle{remark}
\newtheorem{rem}[thm]{Remark}

\makeatletter
\let\c@equation\c@thm
\makeatother
\numberwithin{equation}{section}

\title{Full level structures revisited: pairs of roots of unity}

\author{Preston Wake}

\date{}

\begin{document}
\maketitle

\begin{abstract}
We introduce a notion of full level structure for the group scheme $(\mu_p \times \mu_p) _{/\Z}$, and show that scheme of full level structures is flat over $\Z$. 
\end{abstract}

\section{Introduction}

\subsection{The problem} Let $G$ be a finite, locally free group scheme over a scheme $S$, and suppose that over the locus $S[1/p]$ where $p$ is invertible, $G \times_S S[1/p]$ is \'etale-locally isomorphic to the constant group scheme $(\Z/p^r\Z)^g$. We are interested in finding a closed subscheme of $\uHom_{S}((\Z/p^r\Z)^g,G)$, which we will call the scheme of {\em full homomorphisms}. We will denote this scheme by
$$
\uHom_{S}^\full((\Z/p^r\Z)^g,G) \subset \uHom_{S}((\Z/p^r\Z)^g,G).
$$
The scheme of full homomorphisms should satisfy the following two key properties:
\begin{list}{}{}
\item[1)] There is an equality 
\[\uHom_{S}^\full((\Z/p^r\Z)^g,G) \times_S S[1/p] = \Isom_{S[1/p]}((\Z/p^r\Z)^g,G \times_S S[1/p])
\] 
of closed subschemes of $\uHom_{S[1/p]}((\Z/p^r\Z)^g,G\times_S S[1/p])$.
\item[2)] $\uHom_{S}^\full((\Z/p^r\Z)^g,G)$ is flat over $S$.
\end{list}

The motivation for this problem comes from the study of integral models of Shimura varieties. If $X$ is a Shimura variety over $\Q_p$ with a universal abelian variety $A$ of dimension $g$, then there are interesting covers 
$$
X(p^r) = \Isom_{X}((\Z/p^r\Z)^{2g},A[p^r])
$$
of $X$. If $\X$ is an integral model for $X$, and $\A$ is an extension of $A$ to $\X$, then the problem of defining $\uHom_{\X}^\full((\Z/p^{r}\Z)^{2g},\A[p^r])$ is the same as finding a flat model for the cover $X(p^r)$.

\subsection{Previous results} The most basic case is when $G=\mu_{p^r}$. In this case $\uHom_{S}(\Z/p^r\Z,\mu_{p^r})=\mu_{p^r}(S)$, and $\uHom_{S}^\full(\Z/p^r\Z,\mu_{p^r})$ should be the ``scheme of generators". The condition for a section to be a generator is that it satisfy the cyclotomic polynomial.

Oort and Tate \cite{oort-tate} gave a classification of group schemes of order $p$ (that is, the case when $g=r=1$). In modern terminology, they give an explicit description of the stack of finite flat group schemes of order $p$ and of the universal group scheme over it. Using this description, at least over $\Z_p$, it is possible to define the ``scheme of generators" (see \cite[Theorem 3.3.1]{haines-rapoport} and the paragraph following it). Oort and Tate's work was used by Deligne and Rapoport \cite[Section V.2, pg. 103]{deligne-rapoport} in their study of integral models of modular curves with $\Gamma_1(p)$-level structure -- that is, the moduli space of elliptic curves with a point of order $p$.

More generally, when $G$ is embedded in a curve $C$ over $S$, there is a definition of full homomorphism, based on ideas of Drinfeld \cite{drinfeld}. If $h \in \uHom_{S}((\Z/p^r\Z)^g,G)$ denotes the universal homomorphism, then the scheme $\uHom^\full_{S}((\Z/p^r\Z)^g,G)$ is cut out by the equation
$$
G = \sum_{x \in (\Z/p^r\Z)^g} [h(x)]
$$
of Cartier divisors in $C$. When $G$ has order $p$, this definition coincides with the Oort-Tate definition \cite[Remark 3.3.2]{haines-rapoport}. Drinfeld's definition was used by Katz and Mazur \cite{katz-mazur} to give a very satisfying theory of integral models of modular curves with arbitrary level structure. More recently, it was used by Harris and Taylor \cite{harris-taylor} to study certain special unitary Shimura varieties in the course of their proof of the local Langlands correspondence for $GL_n$. 

If $G$ is not embedded in a curve, then this definition fails: points on $G$ cannot be interpreted as Cartier divisors. Katz and Mazur \cite[Section 1.13]{katz-mazur} introduced a notion of {\em $\times$-homomorphism} of finite flat schemes over $S$ (see Section \ref{subsec: times homom} for a review of this notion). They show that the scheme $\uHom^\times_{S}((\Z/p^r\Z)^g,G)$ of $\times$-homomorphisms agrees with Drinfeld's notion when $G$ is embedded in a curve. They suggest that $\times$-homomorphisms may be useful for studying moduli of abelian varieties of higher dimension \cite[Introduction, pg. xiii]{katz-mazur}. However, Chai and Norman \cite[Section A.2]{chai-norman} prove that, even for the group scheme $G=(\mu_p \times \mu_p) _{/\Zp}$ the scheme of $\times$-homomorphisms is not flat over the base (note that $(\mu_p \times \mu_p) _{/\Zp}$ cannot be embedded in a curve -- see Section \ref{subsec: chai-norman} for an explanation of this). This deficiency was also remarked upon by Pappas \cite[Page 45]{pappas} in his study of level structures for Hilbert modular varieties.

\subsection{Result} We study the group scheme $G=(\mu_p \times \mu_p) _{/\Z}$, and give a definition of $\uHom_{\Z}^\full((\Z/p\Z)^2,\mu_p \times \mu_p)$. Our main result is then the following.

\begin{thm}
\label{thm: main}
The scheme $\uHom_{\Z}^\full((\Z/p\Z)^2,\mu_p \times \mu_p)$ is flat over $\Z$.
\end{thm} 

Our definition of $\uHom_{\Z}^\full$ is given is Section \ref{sec: defn}. It does not use Katz and Mazur's notion of ``full set of sections" or ``$\times$-homomorphism", but it can be related to their notion as follows. Let $h \in \uHom_{\Z}((\Z/p\Z)^2,\mu_p \times \mu_p)$ denote the universal homomorphism. Let $\uHom_{\Z}^\KMD((\Z/p\Z)^2,\mu_p \times \mu_p)$ be the closed subscheme cut out by the conditions that both $h$ {\em and} its Cartier dual $h^\vee$ are $\times$-homomorphisms in the sense of Katz-Mazur. 

\begin{thm}
\label{thm: KMD}
There is an equality
$$
\uHom_{\Z}^\KMD((\Z/p\Z)^2,\mu_p \times \mu_p) = \uHom_{\Z}^\full((\Z/p\Z)^2,\mu_p \times \mu_p)
$$
of closed subschemes of $\uHom_{\Z}((\Z/p\Z)^2,\mu_p \times \mu_p)$.
\end{thm}
\begin{rem}
For general group schemes, the $\KMD$-condition does not give a flat model. See Example \ref{exmp: counter example} for a simple case. The discussion following that example explains why one also should not expect the $\KMD$-type level structures be flat if $G=A[p]$ for a supersingular abelian variety $A$ of dimension $\dim(A)>1$.
\end{rem}

\subsection{Acknowledgments} It's my pleasure to thank Bob Kottwitz for many helpful discussions and for his advice and encouragement throughout this project. His careful reading of an earlier draft led to great improvements in the exposition. I also thank Kazuya Kato for encouraging me to study level structures on abelian varieties. I thank the referee for making some clarifying suggestions.

The results in this paper were first found experimentally using the SAGE mathematics software \cite{sage}, before they were proven. Although none of the proofs in the paper rely on SAGE computations, the proofs could not have been found by me without them. I thank the SAGE developers for providing such a useful tool.

\section{Full homomorphisms}
\label{sec: defn}

In this section we give our definition of $\uHom_{\Z}^\full((\Z/p\Z)^2,\mu_p \times \mu_p)$. It is based on the description of $\GL_2(\F_p)$ as the set of matrices with linearly independent columns -- with two main changes. The first is in the definition of linearly independent. Vectors in a vector space are called linearly independent if any non-trivial linear combination is non-zero. We instead require that any non-trivial linear combination is {\em primitive}, in a certain sense. 

The second difference is that, for matrices, one of the miracles of linear algebra is that the rows are linearly independent if and only if the columns are linearly independent. In our case, this will not be so, and we have to require that both the rows and columns are linearly independent.

\subsection{Primitive vectors}
 We define the notion of a primitive vector in $\mu_p \times \mu_p$. Morally, an element of $\mu_p \times \mu_p$ is {\em primitive} if at least one of the coordinates satisfies the cyclotomic polynomial. More precisely, if 
$$
\cB = \Z[x,y]/(x^p-1,y^p-1)
$$
denotes the coordinate ring of $\mu_p \times \mu_p$, then we define the primitive vectors
$$
(\mu_p \times \mu_p)^\Prim \subset \mu_p \times \mu_p
$$
as the closed subscheme cut out by the ideal generated by $\Phi_p(x)\Phi_p(y)$, where $\Phi_p(X)=1 + X + \dots X^{p-1}$. 

Just as the subscheme of primitive roots of unity is stable under the automorphisms of $\mu_p$, the following lemma shows that $(\mu_p \times \mu_p)^\Prim$ is stable under the automorphisms of $\mu_p\times \mu_p$.

\begin{lem}\label{lem: prim is stable}
The subscheme $(\mu_p \times \mu_p)^\Prim \subset \mu_p \times \mu_p$ is stable under the action of $\Aut(\mu_p \times \mu_p) \cong \GL_2(\F_p)$.
\end{lem}
\begin{proof}
It suffices to show that the ideal $\Phi_p(x)\Phi_p(y) \cB$ is stable under the action of $\GL_2(\F_p)$ on $\cB$. In terms of the basis $\{x^iy^j \ | \ (i,j) \in \F_p^2\}$ of $\cB$ as a $\Z$-module, the action of $\gamma \in \GL_2(\F_p)$ is given by $x^iy^j \mapsto x^{i'}y^{j'}$ where $(i',j')=\gamma(i,j)$. Then we compute
$$
\gamma\Phi_p(x)\Phi_p(y) = \gamma \sum_{(i,j)\in \F_p^2} x^iy^j = \sum_{(i',j')\in \F_p^2} x^{i'}y^{j'}=\Phi_p(x)\Phi_p(y).
$$
\end{proof}

Note that, under any isomorphism $(\mu_p \times \mu_p) \otimes \Z[\zeta_p,1/p] \cong (\Z/p\Z)^2$, the subset   $(\mu_p \times \mu_p)^\Prim \otimes \Z[\zeta_p,1/p]$ is identified with the non-zero vectors.

\subsection{Linear independence} We say that a pair of elements of $\mu_p \times \mu_p$ are {\em linearly independent} if any non-trivial linear combination is primitive. 

\begin{defn}
The subscheme $\uHom^\full_{\Z}((\Z/p\Z)^2,\mu_p \times \mu_p)$ of $\uHom_{\Z}((\Z/p\Z)^2,\mu_p \times \mu_p)$ is the scheme cut out by the conditions that the rows and columns of the universal homomorphism are linearly independent. 
\end{defn}

In order to make this definition explicit, we first introduce some notation. Let $\cA$ be the coordinate ring of $\uHom_{\Z}((\Z/p\Z)^2,\mu_p \times \mu_p) \simeq \mu_p \times \mu_p \times \mu_p \times \mu_p$. We have
$$
\cA = \Z[S,T,U,V]/(S^p-1,T^p-1,U^p-1,V^p-1).
$$
We write the universal homomorphism as
$$
h=\ttmat{S}{T}{U}{V}.
$$
Let $\I \subset \cA$ be the ideal generated by
$$
\{\Phi_p(S^aU^b)\Phi_p(T^aV^b),\Phi_p(S^aT^b)\Phi_p(U^aV^b) \ | \ (a,b)\in \F_p^2 \setminus \{(0,0)\}\}.
$$
Then $\uHom^\full_{\Z}((\Z/p\Z)^2,\mu_p \times \mu_p) = \Spec(\cA/\I).$

\begin{prop}
\label{prop: generic}
Under any isomorphism $(\mu_p \times \mu_p) \otimes_\Z \Z[\zeta_p,1/p] \cong (\Z/p\Z)^2$, we have
$$
\uHom^\full_{\Z}((\Z/p\Z)^2,\mu_p \times \mu_p) \otimes_{\Z} \Z[\zeta_p,1/p]  \cong \GL_2(\F_p)_{/\Z[\zeta_p,1/p]}.
$$
\end{prop}
\begin{proof}
Indeed, the image of $\uHom^\full_{\Z}((\Z/p\Z)^2,\mu_p \times \mu_p)\otimes_{\Z} \Z[\zeta_p,1/p] $ under
$$
\uHom_{\Z}((\Z/p\Z)^2,\mu_p \times \mu_p)\otimes_{\Z} \Z[\zeta_p,1/p] \cong \underline{\End}_{\Z}((\Z/p\Z)^2)
$$
will be the set of matrices whose rows and columns are linearly independent.
\end{proof}

\subsection{Group actions}
\label{subsec: group}
 In our proof of Theorem \ref{thm: main}, we will make use of some group actions on $\cA$. On $\uHom_{\Z}((\Z/p\Z)^2,\mu_p \times \mu_p) = \Spec(\cA)$, there is a right action of $\GL_2(\F_p) = \Aut((\Z/p\Z)^2)$, and a left action of $\GL_2(\F_p) = \Aut( \mu_p \times \mu_p)$, and these two actions commute. This implies that on $\cA$ there is a left action of $\GL_2(\F_p) = \Aut((\Z/p\Z)^2)$, and a right action of $\GL_2(\F_p) = \Aut( \mu_p \times \mu_p)$, and these two actions commute. 

There is an involution on $\cA$ that relates the two actions. Namely, the isomorphism of Cartier duality
$$
\uHom_{\Z}((\Z/p\Z)^2,\mu_p \times \mu_p) \simeq \uHom_{\Z}((\mu_p \times \mu_p)^\vee,((\Z/p\Z)^2)^\vee) \simeq \uHom_{\Z}((\Z/p\Z)^2,\mu_p \times \mu_p)
$$
(where $(-)^\vee$ denotes the Cartier dual) induces an involution $\iota$ on $\cA$. In coordinates, it fixes $S$ and $V$ and it switches $T$ and $U$ -- in other words, $\iota$ is the transpose operator. For $\gamma \in \GL_2(\F_p)$ and $f \in \cA$, we have $\iota(\gamma.f)=\iota(f).\gamma^t$, where $\gamma^t$ is the transpose of $\gamma$.

\section{Flatness}

In this section, we prove Theorem \ref{thm: main}. First notice that we are reduced to proving that the scheme  $\uHom^\full_{\Z}((\Z/p\Z)^2,\mu_p \times \mu_p) \otimes_\Z \Zp$ is flat over $\Zp$. Indeed, by Proposition \ref{prop: generic}, we see that $\uHom^\full_{\Z}((\Z/p\Z)^2,\mu_p \times \mu_p) \otimes_\Z \Z[1/p]$ is \'etale-locally isomorphic to the scheme $\GL_2(\F_p)_{/\Z[1/p]}$. 

Retaining the notation of the previous section, we are reduced to proving that $\cA_{(p)}/\I_{(p)}$ is flat over $\Zp$, or, equivalently, $p$-torsion free. By Proposition \ref{prop: generic}, we already know that 
$$
\dim_{\bar{\Q}}(\cA_{(p)}/\I_{(p)} \otimes_{\Zp} \bar{\Q})= \#\GL_2(\F_p) = (p^2-1)(p^2-p) = p^4-p^3-p^2+p.
$$ 
Then $\dim_{\F_p}(\cA_{(p)}/\I_{(p)} \otimes_{\Zp} \F_p) \ge p^4-p^3-p^2+p$ with equality if and only if $\cA_{(p)}/\I_{(p)}$ is $p$-torsion free. Let $A = \cA_{(p)} \otimes_{\Zp} \F_p$ and let $I$ be the image of $\I_{(p)}$ in $A$. Then we have have $\cA_{(p)}/\I_{(p)} \otimes_{\Zp} \F_p =A/I$, and $\dim_{\F_p}(A)=p^4$, and so we are reduced to showing $\dim_{\F_p}(I) = p^3+p^2-p$. We record this discussion as a lemma.

\begin{lem}
\label{lem: reduction to counting}
We have $\dim_{\F_p}(I) \le p^3+p^2-p$, with equality if and only if $\cA/\I$ is flat over $\Z$.
\end{lem}

\subsection{Column and row generators} To find a lower bound for the dimension of $I$, we will choose nice generators for $I$ and show that there are small intersections between the ideals generated by subsets of these generators. We let $s,t,u,v$ denote the images in $A$ of $S-1,T-1,U-1,V-1 \in \cA$. So we have
$$
A = \F_p[s,t,u,v]/(s^p,t^p,u^p,v^p).
$$

We single out certain elements of $I$ that we will call the \emph{column generators}, that are the images of some of the generators $\Phi_p(S^aT^b)\Phi_p(U^aV^b)$ of $\I$. For $(a,b) \in (\Z/p\Z)^2 \setminus {0}$, we let $\phi_{(a,b)}: \Z/p\Z \to (\Z/p\Z)^2$ denote the homomorphism $1 \mapsto (a,b)$. The element $\Phi_p(S^aT^b)\Phi_p(U^aV^b)$ of $\I$ is cutting out the condition that the element of $\mu_p \times \mu_p$ determined by the composition
$$
\Z/p\Z \xrightarrow{\phi_{(a,b)}} (\Z/p\Z)^2 \xrightarrow{h} \mu_p \times \mu_p
$$
is primitive. Since, by Lemma \ref{lem: prim is stable}, $(\mu_p \times \mu_p)^\Prim \subset \mu_p \times \mu_p$ is stable under the action of $\Aut(\mu_p \times \mu_p)$, we see that the ideal generated by $\Phi_p(S^aT^b)\Phi_p(U^aV^b)$ in $\cA$ is stable under the action of $\Aut(\mu_p \times \mu_p)$. In particular, the ideal generated by $\Phi_p(S^aT^b)\Phi_p(U^aV^b)$ in $\cA$ will be the same if we replace $(a,b)$ by $(\lambda a, \lambda b)$ for any $\lambda \in (\Z/p\Z)^\times$. 

This discussion implies that the ideals $\Phi_p(S^aT^b)\Phi_p(U^aV^b)\A$ can be labeled by classes of $(a,b)$ in $\bP^1(\Z/p\Z)$. We choose the set $\{(1,0),(1,1),\dots, (1,p-1),(0,1)\}$ of representatives for $\bP^1(\Z/p\Z)$, and consider the resulting elements of $I$. That is, for $i=0,1, \dots, p-1$, we let
$$
c_i = \Phi_p((s+1)(t+1)^i)\Phi_p((u+1)(v+1)^i).
$$
Note that, since $\Phi_p(X+1) \equiv X^{p-1} \mod p$, we have $c_0 =(su)^{p-1}$. We also let
$$
c=c_p = \Phi_p(t+1)\Phi_p(v+1)=(tv)^{p-1}.
$$
We call the set $\cC= \{c_0A, c_1A, \dots, c_pA\}$ the set of the column generators. For $J \subset \{0, \dots,p\}$, we let $C(J)=\sum_{i \in J} c_i A$. For $J=\{0, \dots,p\}$, we let $C=C(J)$.

Notice that the elements $c_0$ and $c$ have particularly simple form. The crucial argument in the proof of Theorem \ref{thm: main} will rely on this simple form. In order to reduce to considering only these two simple elements, we will make frequent use of the actions of $\GL_2(\F_p)$ on $A$. We record some nice properties of the column generators under the actions of $\GL_2(\F_p)$ on $A$.

\begin{lem}\label{lem: gl2 on columns} Consider the actions of $\GL_2(\F_p)$ on $A$ as in Section \ref{subsec: group}. 
\begin{list}{}{}
\item[$(1)$] For $i\in \{0, \dots,p\}$, the group $\Aut(\mu_p \times \mu_p)$ stabilizes $c_i A$.
\item[$(2)$] The group $\Aut((\Z/p\Z)^2)$ acts triply transitively on the set $\cC$.
\item[$(3)$] The lower-triangular unipotent subgroup of  $\Aut((\Z/p\Z)^2)$ acts transitively on the set $\{c_0, \dots, c_{p-1}\}$ and trivially on the ideal $cA$.
\end{list}
\end{lem}
\begin{proof}
Statement (1) follows from Lemma \ref{lem: prim is stable}. Statement (2) follows from the fact that $\GL_2$ acts triply transitively on $\bP^1$. 

To see statement (3), let $\gamma = \ttmat{1}{0}{1}{1} \in \Aut((\Z/p\Z)^2)$. Then $\gamma$ acts on $\A$ by
\[\tag{*}
S \mapsto ST, \  T \mapsto T, U \mapsto UV, \  V \mapsto V.
\]
Then it is clear that $c_i. \gamma= c_{i+1}$. This proves the first part of (3). To see that $\gamma$ acts trivially on $cA$, we note that there is a $\Aut((\Z/p\Z)^2)$-equivariant isomorphism
$$
A/\Ann_A(c) \isoto cA
$$
given by $a \mapsto ac$. From the formula $c=(tv)^{p-1}$ we see that $\Ann_A(c)=tA+vA$. It is then clear from (*) that $\Ann_A(c)$ is stable under the action of $\gamma$, and so $\gamma$ acts on $A/\Ann_A(c)$. But $T=V=1$ in $A/\Ann_A(c)$, so we see from (*) that $\gamma$ acts trivially on $A/\Ann_A(c)$, and hence on $cA$.
\end{proof}

We also have \emph{row generators} coming from the generators $\Phi_p(S^aU^b)\Phi_p(T^aV^b)$ of $\I$. Since these elements are obtained from $\Phi_p(S^aT^b)\Phi_p(U^aV^b)$ by interchanging $T$ and $U$, we define $r_i = \iota(c_i)$ for $i=0,\dots, p$, where $\iota$ is the involution on $\A$ defined in Section \ref{subsec: group}. We also define $r=r_p$, and $\cR=\{r_0A, \dots, r_pA\}$. For $J \subset \{0, \dots,p\}$, we let $R(J)=\sum_{i \in J} r_iA$. For $J=\{0, \dots,p\}$, we let $R=R(J)$.

The following lemma is an immediate consequence of Lemma \ref{lem: gl2 on columns} and the relation of the two $GL_2(\F_p)$ actions discussed in Section \ref{subsec: group}.

\begin{lem}\label{lem: gl2 on rows} Consider the actions of $\GL_2(\F_p)$ on $A$ as in Section \ref{subsec: group}. 
\begin{list}{}{}
\item[$(1)$] For $i\in \{0, \dots,p\}$, the group $\Aut((\Z/p\Z)^2)$ stabilizes $r_i A$.
\item[$(2)$] The group $\Aut(\mu_p \times \mu_p)$ acts triply transitively on the set $\cR$.
\item[$(3)$] The upper-triangular unipotent subgroup of $\Aut(\mu_p \times \mu_p)$ acts transitively on the set $\{r_0, \dots, r_{p-1}\}$ and trivially on the ideal $rA$.
\end{list}
\end{lem}

\subsection{The proof of flatness} We have seen in Lemma \ref{lem: reduction to counting} that the proof of Theorem \ref{thm: main} is reduced to counting the dimension of $I$. In fact, we do more: we count the dimensions of ideals generated by just some of the column and row generators. The main result is a set of formulas for the dimension of ideals generated by column and row elements. In the statement of the theorem, we use the binomial coefficients
\[
{n \choose k} = \prod_{i=1}^{k} \frac{n+1-i}{i},
\]
where $n$ and $k$ are integers and $k \ge 1$.
\begin{thm}
\label{thm:dim formulas}
 Let $J \subset \{0, \dots, p+1\}$ and let $k=\# J$. The dimension of the vector spaces $C(J)$, $R(J)$,  $C+R(J)$, and $C(J)+R$ depend only on $k$ and not on $J$. We have the explicit formulas:
\begin{list}{}{}
\item[$(1)$] $\displaystyle \dim_{\F_p}(C(J))= \dim_{\F_p}(R(J))=kp^2 - {k+1 \choose 3}$
\item[$(2)$] $\displaystyle \dim_{\F_p}(C+R(J))=\dim_{\F_p}(C(J)+R)=p^3+p^2-p - {p-k+2 \choose 3}$
\end{list}
\end{thm}

Before proceeding with the proof of this theorem, we record some corollaries.
\begin{cor}
\label{cor: main refined}
For any subset $J \subset \{0, \dots, p+1\}$ with $\# J= p$, we have $C+R(J)=C(J)+R=I$, and
$$
\dim_{\F_p} (I)= p^3+p^2-p.
$$
\end{cor}
\begin{proof}
We have $\dim_{\F_p}(C+R(J))=\dim_{\F_p}(C(J)+R)=p^3+p^2-p$ by the theorem. But $C+R(J), C(J)+R \subset I$ by definition, and $\dim_{\F_p} (I) \le  p^3+p^2-p$ by Lemma \ref{lem: reduction to counting}, so the corollary follows.
\end{proof}
By Lemma \ref{lem: reduction to counting}, this corollary proves Theorem \ref{thm: main}. The following is also now immediate.
\begin{cor}
If $J \subset \{0, \dots, p+1\}$ with $\# J< p$, then $C+R(J) \subsetneq I$ and $C(J)+R \subsetneq I$. In particular, $C \ne I$ and $R \ne I$.
\end{cor}

The proof of Theorem \ref{thm:dim formulas} relies on the following proposition.

\begin{prop}
\label{prop:dimension inequalities}
Let $J \subset \{0, \dots p+1\}$ be a proper subset, and let $k = \# J+1$. Let $i \in \{0, \dots p+1\}$ be some element not in $J$. Then 
\begin{list}{}{}
\item[$(1)$] $\displaystyle \dim_{\F_p}\left(\frac{C(J)+c_iA}{C(J)}\right) \ge p^2 - {k \choose 2}$.
\item[$(2)$] $\displaystyle \dim_{\F_p}\left(\frac{C+R(J)+r_iA}{C+R(J)}\right) \ge \left \{ \begin{array}{lr}
\displaystyle {p \choose 2} &  \mathrm{if } \  k=1 \\
\displaystyle {p-k+2 \choose 2} &  \mathrm{if } \  k > 1	
\end{array}. \right.
$
\end{list}
\end{prop}

\begin{proof}[Proof of Theorem \ref{thm:dim formulas} assuming Proposition \ref{prop:dimension inequalities}]
We first prove that the inequalities in the proposition are equalities. Let $\sigma, \tau$ be any permutations of the set $\{0,\dots, p+1\}$. Then there is an associated increasing filtration
$$
0 = F_0 \subset \dots \subset F_{2p+2} = C+R
$$
on $C+R$ given by 
$$
F_n= \left \{ \begin{array}{lr}
\displaystyle \sum_{i=0}^{n-1} c_{\sigma(i)} A &  \mathrm{if } \ n \le p+1 \\
\displaystyle C+\sum_{i=0}^{n-p-2} r_{\tau(i)}A
 &  \mathrm{if } \  n> p+1	
\end{array}. \right. 
$$
Proposition \ref{prop:dimension inequalities} then gives lower bounds on the dimensions of the graded pieces of the filtration $F_\bullet$:
$$
 \dim_{\F_p} \gr_i^F(C+R) \ge \left \{ \begin{array}{lr}
\displaystyle p^2 - {i \choose 2} &  \mathrm{if } \ i \le p+1 \\
\displaystyle {p \choose 2} &  \mathrm{if } \  i= p+2 \\
\displaystyle {p-(i-p-1)+2 \choose 2} = {2p+3-i \choose 2}&  \mathrm{if } \  i>p+2 \\
\end{array} \right. .
$$
In particular, we see that, for $i=1, \dots p$, we have
$$
\dim_{\F_p} \gr_i^F(C+R) +\dim_{\F_p} \gr_{2p+3-i}^F(C+R) \ge p^2,
$$
and for $i=p+1$, we have
$$
\dim_{\F_p} \gr_{p+1}^F(C+R) +\dim_{\F_p} \gr_{p+2}^F(C+R) \ge p^2-{p+1 \choose 2} + {p \choose 2} = p^2 -p.
$$
We obtain the inequality
\begin{align*}\tag{*}
\dim_{\F_p}(C+R) &= \sum_{i=1}^{2p+2} \dim_{\F_p} \gr_i^F(C+R) \\
&= \sum_{i=1}^{p+1} \left(\dim_{\F_p} \gr_i^F(C+R) +\dim_{\F_p} \gr_{2p+3-i}^F(C+R) \right)\\
& \ge p^3 +p^2-p.
%
%
\end{align*}
On the other hand, since $C+R \subset I$, we have $\dim_{\F_p}(C+R) \le p^3+p^2-p$ by Lemma \ref{lem: reduction to counting}. It follows that $\dim_{\F_p}(C+R) = p^3+p^2-p$, and that the inequality (*) is an equality. Since (*) was obtained as the sum of the inequalities in Proposition \ref{prop:dimension inequalities}, it follows that those inequalities are equalities.

The dimension formulas for $C(J)$ and $C+R(J)$ then follow by truncating the filtration $F_\bullet$ and using the identity
$$ 
\displaystyle \sum_{i=0}^b {a \choose i}={a+1 \choose b+1}.
$$
The formulas for $R(J)$ and $C(J)+R$ follow by applying the involution $\iota$ from Section \ref{subsec: group}, since $\iota: R(J) \isoto C(J)$ and $\iota: C(J)+R  \isoto C+R(J)$.
\end{proof}

\subsection{Key argument} We now turn to the proof of Proposition \ref{prop:dimension inequalities}. For this, we need to compute intersections of ideals. Intersections of ideals in $A$ can be quite complicated to compute in general. However, using the actions of $\GL_2(\F_p)$, we will see that it is enough to compute some intersections with the principal ideals $cA$ and $rA$. 

As in the proof of Lemma \ref{lem: gl2 on columns}, there are canonical isomorphisms
$$
A/\Ann_A(c) \simeq cA, \text{ and } A/\Ann_A(r) \simeq rA.
$$
Let $A_{(s,u)}$ and $A_{(s,t)}$ be the smallest subrings of $A$ containing $s$ and $u$ and $s$ and $t$, respectively, and let $\m_{(s,u)}$ and $\m_{(s,t)}$ denote their respective maximal ideals. Let $B=\F_p[x,y]/(x^p,y^p)$, and let $\m_B$ be its maximal ideal. We have isomorphisms of local rings $(A_{(s,u)},\m_{(s,u)}) \cong (A_{(s,t)},\m_{(s,t)}) \cong (B, \m_B)$.

Since $\Ann_A(c)=tA+vA$ and $\Ann_A(r)=uA+vA$, we see that the composite maps
$$
A_{(s,u)} \to A/\Ann_A(c), \text{ and } A_{(s,t)} \to A/\Ann_A(r)
$$
are isomorphisms. In particular, for any ideal $\fI \subset A_{(s,u)}$ of $A_{(s,u)}$, the subset $c \fI \subset A$ is an ideal of $A$. Similarly, for any ideal $\fI \subset A_{(s,t)}$, the subset $r \fI \subset A$ is an ideal.

\begin{prop}
\label{prop: intersections}
Let $k$ be an integer with $1 \le k \le p$. If $J \subset \{0,1, \dots p-1\}$ and $\# J = k$, then
\begin{list}{}{}
\item[$(1)$] $C(J) \cap c A \subset c\m_{(s,u)}^{2p-k-1}$
\item[$(2)$] $(C+R(J)) \cap r A \subset r\m_{(s,t)}^{p-k}.$
\end{list}
\end{prop}
\begin{proof}[Proof of Proposition \ref{prop:dimension inequalities} assuming Proposition \ref{prop: intersections}] For the proof of Proposition \ref{prop:dimension inequalities} we may assume that $i=p$. Indeed, using Lemma \ref{lem: gl2 on columns} we see that we can use the $\Aut((\Z/p\Z)^2)$-action to move $c_i$ to $c$, and using Lemma \ref{lem: gl2 on rows} we see that we can use the $\Aut(\mu_p\times \mu_p)$-action to move $r_i$ to $r$ without changing $C$.

We have
$$
\dim_{\F_p}\left(\frac{C(J)+cA}{C(J)} \right)=\dim_{\F_p}(cA)- \dim_{\F_p}(C(J)\cap cA)=p^2- \dim_{\F_p}(C(J) \cap c A)
$$ 
and so Proposition \ref{prop:dimension inequalities} (1) follows from Proposition \ref{prop: intersections} (1), since it is clear that
$$
\dim_{\F_p}(c\m_{(s,u)}^{2p-k-1}) = \dim_{\F_p}(\m_B^{ 2p-k-1}) = {k+1 \choose 2}
$$
for $1 \le k \le p$. Similarly, Proposition \ref{prop:dimension inequalities} (2) follows from Proposition \ref{prop: intersections} (2), since
$$
\dim_{\F_p}(r\m_{(s,t)}^{p-k}) = \dim_{\F_p}(\m_B^{ p-k}) = p^2 - {p-k+1 \choose 2}
$$
for $1 \le k \le p$.
\end{proof}

The next lemma is the key point of the argument.

\begin{keylem}
Let $\fI \subset A$ be an ideal, and suppose that
$$
\fI \cap c A \subset c \m_{(s,u)}^ m
$$
for some integer $m$. Then
$$
(\fI+c_0A) \cap c A \subset c\m_{(s,u)}^ {m-1}
$$
\end{keylem}

\begin{proof}
Fix an element $f \in (\fI + c_0 A)  \cap c A.$
Since $cA \simeq A/\Ann_A(c)\simeq A_{(s,u)}$, we see that $f=ac$ for a unique element $a \in A_{(s,u)}$. We will show that $a \in \m_{(s,u)}^{m-1}$.  For this, we will use the following fact about the ring $B$.

\begin{lem}
\label{lem: division}
For any positive integer $d$ with $d \le 2p+3$, we have $\Ann_B(\m_B/\m_B^d) = \m_B^{d-1}.$
\end{lem}
\begin{proof}
Note that, for any $n \ge 0$ the set $\{x^iy^j \ | \ i+j \ge n\}$ forms a basis for $\m_B^n$ as an $\F_p$-vector space. Clearly, we have $\Ann_B(\m_B/\m_B^d) \supset \m_B^{d-1}$; we show the opposite inclusion. Let $f \in \Ann_B(\m_B/\m_B^d)$, and write
$$
f = \sum_{i,j=0}^{p-1} \alpha_{i,j} x_i y^j.
$$
By subtracting an element of $\m_B^{d-1}$, we may assume $\alpha_{i,j} =0$ if $i+j \ge d-1$. Moreover, since $xf \in \m_B^d$, we see that $\alpha_{i,j}=0$ unless $i=p-1$. Similarly, since $yf \in \m_B^d$, we see that $\alpha_{i,j}=0$ unless $j=p-1$. But $\alpha_{p-1,p-1}=0$ since $2p+2 \ge d-1$ by assumption. Hence $f=0$.
\end{proof}
 From the lemma, we see that it suffices to show that, for any $z \in \m_{(s,u)}$, we have $za \in \m_{(s,u)}^{m}$. Let $z \in \m_{(s,u)}$ be arbitrary. Since $zf=zac$, we are reduced to showing that $zf \in c\m_{(s,u)}^{m}$.
 
Now write $f=g+h$, with $g \in \fI$ and $h \in c_0A$. As $\m_{(s,u)} \subset \Ann_A(c_0A)$, we have $zf=zg$. Since $zf \in cA$ and $zg \in \fI$, we have $zf=zg \in \fI\cap cA$. Then, by assumption $zf \in c\m_{(s,u)}^{m}$. This completes the proof.
\end{proof}

We can also switch the roles of $u$ and $t$ by applying the involution $\iota$ from Section \ref{subsec: group} to the Key Lemma.
\begin{cor}
\label{cor:key cor}
Let $\fI \subset A$ be an ideal, and suppose that
$$
\fI \cap r A \subset r \m_{(s,t)}^m
$$
for some integer $m$. Then
$$
(\fI+r_0A) \cap r A \subset r\m_{(s,t)}^{m-1}
$$
\end{cor}

We can now prove Proposition \ref{prop: intersections}.

\begin{proof}[Proof of Proposition \ref{prop: intersections}]
For both parts, the proof is by induction on $k$. 

\textbf{Base case (1):} We first assume $k=1$. Without loss of generality, we may assume $J=\{0\}$. Indeed, by Lemma \ref{lem: gl2 on columns} (3), the lower triangular unipotent of $\Aut((\Z/p\Z)^2)$ acts transitively on $\{c_0, \dots, c_{p-1}\}$ and acts trivially on $cA$ (and hence stabilizes the submodule $c\m_{(s,u)}^{2p-2}).$

For $J=\{0\}$ we have
$$
C(J) \cap c A = c_0 A \cap c A = (su)^{p-1}A \cap (tv)^{p-1}A = (stuv)^{p-1}A.
$$
As this is the unique minimal ideal of $A$, it is clearly contained in $c\m_{(s,u)}^{2p-2}.$ This proves the base case.

\textbf{Inductive step (1):} We now assume the proposition is proved for $k$ and prove it for $k+1$. As above, we may assume $J=J'\cup \{0\}$ where $\#J'=k$. Indeed, if $0 \not \in J$, then we may act by the lower triangular unipotent, which stabilizes $c\m_{(s,u)}^{2p-(k+1)-1}$.

The induction hypothesis gives $C(J') \cap cA \subset c\m_{(s,u)}^{2p-k-1}$, and so 
\[
C(J) \cap cA = (C(J')+c_0A) \cap cA \subset c\m_{(s,u)}^{2p-(k+1)-1}
\]
 by the Key Lemma. This completes the proof of (1).

\textbf{Base case (2):} We first assume $k=1$. As above, we reduce to the case $J=\{0\}$: by Lemma \ref{lem: gl2 on rows} (3), the upper triangular unipotent in $\Aut(\mu_p \times \mu_p)$ acts transitively on $\{r_0, \dots, r_{p-1}\}$ and trivially on $rA$, and hence stabilizes $r\m_{(s,t)}^{p-1}$; moreover, by Lemma \ref{lem: gl2 on columns} (1), it stabilizes $C$.

To show $(C+r_0A)\cap rA \subset r\m_{(s,t)}^{p-1}$, we first note that
$$
C+r_0A \subset \m_{(s,t)}^{p-1}A,
$$
where $\m_{(s,t)}^{p-1}A$ denotes the ideal in $A$ generated by $\m_{(s,t)}^{p-1}$. Indeed, it is clear that $\m_{(s,t)}^{p-1}A$ is the $\F_p$-subspace spanned by the elements $s^it^ju^kv^l$ with $i+j \ge p-1$, and so it is enough to show that $r_0$ and $c_i$ for $i=0, \dots p$ have degree greater than or equal to $p-1$ in $s$ and $t$. For $r_0=(st)^{p-1}$, $c_0=(su)^{p-1}$ and $c=(tv)^{p-1}$, this is clear. For $c_i = ((s+1)(t+1)^i-1)^{p-1}((u+1)(v+1)^i-1)^{p-1}$, this follows from the fact that $(s+1)(t+1)^i-1$ has no constant term.

We now have
$$
(C+r_0A)\cap rA \subset \m_{(s,t)}^{p-1}A\cap rA.
$$
But $\m_{(s,t)}^{p-1}A\cap rA= r\m_{(s,t)}^{p-1}$, as can be seen by computing in the standard basis, as in the proof of Lemma \ref{lem: division}. This proves the base case for (2). 

\textbf{Inductive step (2):} We now assume the proposition is proved for $k$ and prove it for $k+1$. As above, we may assume $J=J'\cup \{0\}$ where $\#J'=k$. Indeed, if $0 \not \in J$, then we may act we may act by the lower triangular unipotent, which stabilizes  $r\m_{(s,t)}^{p-(k+1)}$.

Then $(C+R(J')) \cap rA \subset r\m_{(s,t)}^{p-k}$ by the induction hypothesis, and so 
\[
(C+R(J)) \cap rA  =( C+R(J')+r_0A) \cap rA \subset r\m_{(s,t)}^{p-(k+1)}
\]
 by Corollary \ref{cor:key cor} of the Key Lemma. This completes the proof of (2).
\end{proof}

\section{Comparison with Drinfeld-Katz-Mazur Level Structures}

In this section, we compare our notion to full homomorphism to the one used by Katz and Mazur. 

\subsection{Full set of sections and $\times$-homomorphism}
\label{subsec: times homom}
We recall the notions of full set of sections and $\times$-homomorphism, following \cite[Section 1.8, pg. 32]{katz-mazur}. Let $S$ be a scheme, and let $Z_{/S}$ be a finite flat scheme of finite presentation and of rank $N$. This implies that, for any $\Spec(R) \to S$, we have $Z_R:=Z \times_S \Spec(R) = \Spec(B)$ where $B$ is an $R$-algebra that is locally free of rank $N$ as an $R$-module. In particular, for $f \in \End_R(B)$ (for example $f \in B$), we can consider $\det_R(f) \in R$. 

Now let $Z'_{/S}$ be a finite flat scheme of finite presentation of rank $N$ and let $\phi: Z' \to Z$. We say that $\phi$ is a {\em $\times$-homomorphism} if for any $\Spec(R) \to S$ and any $f \in H^0(\sO_{Z_R})$ , we have an equality in $R[T]$
\begin{equation}\label{eq: times homom}
\det(T-f) = \det(T-\phi^*(f)).
\end{equation}

The set of $\times$-homomorphisms from $Z'$ to $Z$ is denoted by $\uHom_S^{\times}(Z',Z)$. It is a closed subscheme of the $S$-scheme $\uHom_S(Z',Z)$.

If $Z'$ is \'etale (and so $Z' \cong \coprod_{i=1}^N S$), then giving a $\phi: Z' \to Z$ is equivalent to giving $P_1, \dots, P_N \in Z(S)$; in this case, if $\phi$ is a $\times$-homomorphism, we say that  $P_1, \dots, P_N \in Z(S)$ is a {\em full set of sections}.

Note that if $\phi$ is an isomorphism, then it is a $\times$-homomorphism. Note also that the composition of two $\times$-homomorphisms is a $\times$-homomorphism. If $Z'$ and $Z$ are \'etale, then $\phi : Z' \to Z$ is a $\times$-homomorphism if and only if it is an isomorphism \cite[Lemma 1.8.3, pg. 33]{katz-mazur}.

\subsection{Deficiencies}
\label{subsec: chai-norman}
The notion of full set of sections is well-behaved when the scheme $Z$ is embedded in a curve $C$. In this case, a set of sections $P_i$ is full if and only if it gives an equality of Cartier divisors $\sum [P_i] = [Z]$ in $C$ \cite[Theorem 1.10.1]{katz-mazur}.

On the other hand, in \cite[Appendix]{chai-norman}, many examples are presented to demonstrate the deficiencies of the notion of full set of sections. In particular, they show that the schemes 
$$
\uHom_{\Zp}^\times(\mu_p, \mu_p) \text{ and } \uHom_{\Zp}^\times((\Z/p\Z)^2, (\mu_p)^2)
$$
are not flat over $\Spec(\Zp)$. 

Note that this does not contradict the fact that full set of sections is well-behaved when the scheme $Z$ is embedded in a curve $C$: the first example is not a ``full set of sections'' because $\mu_p$ is not \'etale, and, while the second example is a full set of sections, the scheme $(\mu_p)^2_{/\Zp}$ cannot be embedded in a curve over $\Zp$. Indeed, a curve over $\Zp$ has Krull dimension $2$, and one sees that the coordinate ring $\Zp[x,y]/(x^p-1,y^p-1)$ of $(\mu_p)^2_{/\Zp}$, has $\{x-1,y-1,p\}$ as a minimal set of generators for its maximal ideal, and so $(\mu_p)^2_{/\Zp}$ has embedding dimension $3$ (see, for example, \cite[Section 2.3, pg. 73]{bruns-herzog} for a discussion of embedding dimension).
\subsection{Cartier duality}
If $G$ is a finite flat group scheme over $S$, let $G^\vee$ denote its Cartier dual. If $\phi: G \to H$ is a morphism of finite flat group schemes over $S$, let $\phi^\vee : H^\vee \to G^\vee$ denote the Cartier dual. This induces an isomorphism
$$
\uHom_S(G,H) \simeq \uHom_S(H^\vee,G^\vee).
$$
\begin{lem}
Under the isomorphism
$$
\uHom_S(G,H) \simeq \uHom_S(H^\vee,G^\vee),
$$
the closed subschemes $\uHom^\times_S(G,H)$ and $\uHom^\times_S(H^\vee,G^\vee)$, do {\em not}, in general, coincide.
\end{lem}
\begin{proof}
By Chai and Norman's example, we see that $\uHom_{\Zp}^\times(\mu_p, \mu_p)$ is not flat over $\Zp$. But we have $\mu_p^\vee = \Z/p\Z$, and we know that 
$$
\uHom^\times_{\Zp}(\Z/p\Z, \Z/p\Z) = \Aut_{\Zp}(\Z/p\Z) = (\Z/p\Z)^\times_{/\Zp} 
$$
is flat over $\Zp$. 
\end{proof}

We see that the notion of $\times$-homomorphism does not respect Cartier duality. This leads us to the following definition.

\begin{defn}
Let $G$ and $H$ be finite flat group schemes over $S$. By the isomorphism
$$
\uHom_S(G,H) \cong \uHom_S(H^\vee,G^\vee),
$$
we may consider $\uHom^\times_S(H^\vee,G^\vee)$ as a closed subscheme of $\uHom_S(G,H)$. We define
$$
\uHom^\KMD_S(G,H) = \uHom_S^\times(G,H) \cap \uHom^\times_S(H^\vee,G^\vee).
$$
In particular, $\uHom^\KMD_S(G,H)$ is a closed subscheme of $\uHom_S(G,H)$.
\end{defn}

\begin{rem}\label{rem: etale KMD}
If $G$ and $H$ are \'etale groups schemes on $S$, then
$$
\uHom^\KMD_S(G,H) = \Isom_S(G,H).
$$
Indeed, the $\times$-homomorphisms between \'etale schemes are exactly the isomorphisms \cite[Lemma 1.8.3, pg. 33]{katz-mazur}, and the Cartier dual of an isomorphism is an isomorphism, and so, in particular, a $\times$-homomorphism.
\end{rem}

\subsection{Comparison} In this section, we prove Theorem \ref{thm: KMD}.  Let $\I^\KMD$ be the ideal of $\A$ such that
$$
\uHom^\KMD_\Z((\Z/p\Z)^2,\mu_p \times \mu_p) = \Spec(\A/\I^\KMD).
$$
The main thing is to prove the following.
\begin{prop}
\label{prop: map from KMD}
There is an inclusion of ideals $\I \subset \I^\KMD$.
\end{prop}
\begin{proof}
By the definition of $\I$, it suffices to prove that the column and row generators are in $\I^\KMD$. We use the following lemma.

\begin{lem}
The ideal $\I^\KMD \subset \A$ is stable under the left and right $\GL_2(\F_p)$ actions on $\A$.
\end{lem}
\begin{proof}
It suffices to show that the actions on $\A$ descend to actions on $\A/\I^\KMD$, or, equivalently, that the action on $\uHom_\Z((\Z/p\Z)^2,\mu_p \times \mu_p)$ restrict to actions on $\uHom^\KMD_\Z((\Z/p\Z)^2,\mu_p \times \mu_p)$. But this is clear, since isomorphisms are $\times$-homomorphisms and $\times$-homomorphisms are closed under composition.
\end{proof}

Now, to prove the proposition, it suffices to show that $\Phi_p(S)\Phi_p(U), \Phi_p(S)\Phi_p(T) \in \I^{\KMD}$. Indeed, since the other column and row generators are in the $\GL_2(\F_p)$-orbits of these two, the lemma will then give us the proposition.

Now, we let $R^\KMD$ denote the ideal cutting out $\uHom^\times_\Z((\Z/p\Z)^2,\mu_p \times \mu_p)$. We will first show that $\Phi_p(S)\Phi_p(T) \in R^\KMD$.

In order to do this, we give a more explicit description of some elements of the ideal $C^\KMD$. We let $h$ denote the universal homomorphism
$$
h=\ttmat{S}{T}{U}{V}: (\Z/p\Z)^2_{/\A} \to (\mu_p \times \mu_p)_{/\A};~ (1,0) \mapsto (S,U), (0,1) \mapsto (T,V).
$$
The induced map on coordinate rings is
$$
h^*: {\A[X,Y]}/{(X^p-1,Y^p-1)} \to \A^{(\Z/p\Z)^2}
$$
given by
$$
X \mapsto \sum_{i,j=0}^{p-1} S^iT^j e_{(i,j)}
$$
$$
Y \mapsto \sum_{i,j=0}^{p-1} U^iV^j e_{(i,j)},
$$
where $e_{(i,j)}$ is the standard basis vector in $ \A^{(\Z/p\Z)^2}$ corresponding to $(i,j)\in (\Z/p\Z)^2$. 

From equation \eqref{eq: times homom}, we see that $\Tr(h^*(a))-\Tr(a) \in R^\KMD$ for any $a \in \A[X,Y]/(X^p-1,Y^p-1)$.  Consider the element $X \in \A[X,Y]/(X^p-1,Y^p-1)$. A simple calculation shows that $\Tr(X)=0$. On the other hand, we compute
\[
\Tr(h^*(X))= \Tr \left(\sum_{i,j=0}^{p-1} S^iT^j e_{(i,j)}\right) = \sum_{i,j=0}^{p-1} S^iT^j = \Phi_p(S)\Phi_p(T).
\]
This shows that $\Phi_p(S)\Phi_p(T) \in R^\KMD$.

Now we let $C^\KMD$ denote the ideal cutting out $\uHom^\times_\Z((\mu_p \times \mu_p)^\vee,((\Z/p\Z)^2)^\vee)$, thought of as a subscheme of $\uHom_\Z((\Z/p\Z)^2,\mu_p \times \mu_p)$ via the isomorphism 
\begin{equation*}\tag{*}
\uHom_\Z((\Z/p\Z)^2,\mu_p \times \mu_p) \simeq \uHom_\Z((\mu_p \times \mu_p)^\vee,((\Z/p\Z)^2)^\vee).
\end{equation*}
Then $C^\KMD$ admits the same explicit description as $R^\KMD$, except that we replace the universal homomorphism $h$ with its Cartier dual $h^\vee$. A simple computation shows that, under the isomorphism (*), $h^\vee$ is
$$
h^\vee =\ttmat{S}{U}{T}{V}: (\Z/p\Z)^2_{/\A} \to (\mu_p \times \mu_p)_{/\A}.
$$
That is, $h^\vee$ is obtained from $h$ by reversing the roles of $T$ and $U$. Then, since we have shown $\Phi_p(S)\Phi_p(T) \in R^\KMD$, it follows formally that $\Phi_p(S)\Phi_p(U) \in C^\KMD$.

By definition, $\I^\KMD = R^\KMD+C^\KMD$, and so we have $\Phi_p(S)\Phi_p(T), \Phi_p(S)\Phi_p(U) \in \I^\KMD$, which completes the proof.
\end{proof}
\begin{rem}
The proof shows that $R \subset R^\KMD$ and $C \subset C^\KMD$. We suspect that these inclusions are equalities, and checked by hand that this is true for $p=2,3$. We were unable to find a proof in general.
\end{rem}

We can now prove Theorem \ref{thm: KMD}.

\begin{thmB}
There is an equality
$$
\uHom_{\Z}^\KMD((\Z/p\Z)^2,\mu_p \times \mu_p) = \uHom_{\Z}^\full((\Z/p\Z)^2,\mu_p \times \mu_p)
$$
of closed subschemes of $\uHom_{\Z}((\Z/p\Z)^2,\mu_p \times \mu_p)$.
\end{thmB}
\begin{proof}
By Proposition \ref{prop: map from KMD}, there is a surjective $\A$-algebra homomorphism
$$
\phi: \A/\I \onto \A/\I^\KMD.
$$
Moreover, by Proposition \ref{prop: generic} and Remark \ref{rem: etale KMD}, we have that $\phi \otimes_{\Z} \Z[1/p]$ is an isomorphism. In particular, $\ker(\phi)$ is $p$-torsion. But by Theorem \ref{thm: main}, $\A/\I$ is flat as a $\Z$-algebra, and so it has no $p$-torsion. Hence $\ker(\phi)=0.$ 
\end{proof}

Finally, we give an example to show that the $\KMD$-type level structure does not give a flat space in general.

\begin{exmp}
\label{exmp: counter example}
Let $p=2$ and let $S=\mathbb{A}^1_{/\F_2}=\Spec(\F_2[t])$ and let $S^* = S \setminus \{0\}$. Let $H=\Spec(A)$ be the group over $S$ given as follows. The ring $A$ is $A=\F_2[t][X]/(X^2)$, and the comultiplication $A \to A \otimes_{\F_2[t]} A$ is given by $X \mapsto X \otimes 1 + 1 \otimes X + t X \otimes X$. The generic fiber is $H \times_S S^* \cong \mu_2$ and the special fiber is $H \times_S \{0\} \cong \alpha_2$ (c.f. \cite[Section 1]{oss}).

Now let $G=H \times H$. We'll show that $X=\uHom_S^\KMD((\Z/2\Z)^2,G)$ is not flat over $S$. By Theorem \ref{thm: KMD}, we have that $X \times_S S^*$ is locally free of rank $\# \GL_2(\F_2)=6$ over $S^*$. So it suffices to show that $X \times_S \{0\} = \uHom^\KMD_{\F_2}((\Z/2\Z)^2,\alpha_2^2)$ has rank different from $6$.

We first compute $\uHom^\times_{\F_2}((\Z/2\Z)^2,\alpha_2^2)$. We have $\uHom_{\F_2}((\Z/2\Z)^2,\alpha_2^2) =\alpha_2^4$, which we write as $\Spec(B)$ where $B=\F_2[x_1,y_1,x_2,y_2]/(x_1^2,y_1^2,x_2^2,y_2^2)$, and we write the universal homomorphism as
\[
h: (\Z/2\Z)^2_{/B} \to (\alpha_2^2)_{/B}, ~ (1,0) \mapsto (x_1,y_1),~(0,1) \mapsto (x_2,y_2).
\]
In terms of coordinate rings, this is
\[
h^*: B[X,Y]/(X^2,Y^2) \to B^{(\Z/2\Z)^2}, ~ X \mapsto \sum_{(i,j)} (ix_1+jx_2)e_{(i,j)}, ~ Y \mapsto \sum_{(i,j)} (iy_1+jy_2)e_{(i,j)},
\]
where we retain the notation $e_{(i,j)}$ as above. If we let $f=\sum_{(k,l)} b_{kl} X^kY^l \in B[X,Y]/(X^2,Y^2)$ denote a generic section, 
then the condition that $h$ be a $\times$-homomorphism is given by the equality
\[
\det(f)=\det(h^*(f)),
\]
thought of as an equation of polynomials in $B[b_{kl}]$. We see that $\det(f)=b_{00}^4$, and compute that the set of coefficients of $\det(h^*(f))-b_{00}^4$ is 
\[
\{y_1y_2,x_1x_2, x_1y_2+x_2y_1, x_1x_2y_1y_2\}.
\] 
It follows that the ideal in $B$ cutting $\uHom^\times_{\F_2}((\Z/2\Z)^2,\alpha_2^2)$ is generated by this set. A quick computation shows that this ideal has dimension $8$ over $\F_2$. Since $\dim_{\F_2}(B)=16$, we see that $\uHom^\times_{\F_2}((\Z/2\Z)^2,\alpha_2^2)$ has rank $8=16-8$ over $\Spec(\F_2)$.

Next we compute the ideal defining $\uHom^\times_{\F_2}((\alpha_2^2)^\vee,((\Z/2\Z)^2)^\vee)$. The map that $h^\vee$ induces on coordinate rings a homomorphism
\[
(h^\vee)^*: B[S,T]/(S^2-1,T^2-1) \to B[X,Y]/(X^2,Y^2).
\]
Writing $s=S-1$ and $t=T-1$, we can write the universal section as $g= \sum_{(k,l)} c_{kl} s^kt^l \in B[S,T]/(S^2-1,T^2-1)$. However, noting that $(h^\vee)^*(g) \in c_{00} + (X,Y)B[X,Y]/(X^2,Y^2)$, one sees immediately that 
\[
\det(g) = c_{00}^4 = \det((h^\vee)^*(g)).
\]
Hence \emph{there are no conditions} for $h^\vee$ to be a $\times$-homomorphism and 
\[
\uHom^\times_{\F_2}((\alpha_2^2)^\vee,((\Z/2\Z)^2)^\vee) = \uHom_{\F_2}((\alpha_2^2)^\vee,((\Z/2\Z)^2)^\vee).
\]

So we see that $X \times_S \{0\}=\uHom^\KMD_{\F_2}((\Z/2\Z)^2,\alpha_2^2) = \uHom^\times_{\F_2}((\Z/2\Z)^2,\alpha_2^2)$ has rank $8$ over $\Spec(\F_2)$. Since $X$ has generic rank $6$, this implies that $X$ is not flat over $S$.
\end{exmp}

In the example, we chose $p=2$ for simplicity of computation -- the same example would work for any $p$. Other, more complicated, examples may be found over over base schemes than $S=\mathbb{A}^1_{/\F_p}$ (such as $S=\Spec(\Z_{(p)})$) -- the main point is the following. For the same reason that, in the example, $(h^\vee)^*$ is automatically a $\times$-homomorphism, any homomorphism between connected group schemes over a field will be a $\times$-homomorphism. This means that, for groups that are residually connected-connected type (for example, torsion subgroups of supersingular abelian varieties), adding the ``dual'' condition will not add any conditions that survive modulo $p$. To define a good notion of level structures for such groups, a new idea is needed.

Department of Mathematics, UCLA, Los Angeles, CA 90095-1555

email: wake@math.ucla.edu

\begin{thebibliography}{9}

\bibitem[B-H]{bruns-herzog}
Bruns, W.; Herzog, J.
Cohen-Macaulay rings. 
Cambridge Studies in Advanced Mathematics, 39. Cambridge University Press, Cambridge, 1993. xii+403 pp. ISBN: 0-521-41068-1

\bibitem[C-N]{chai-norman}
Chai, C; Norman, P.
Bad reduction of the Siegel moduli scheme of genus two with $\Gamma_0(p)$-level structure. 
Amer. J. Math. 112 (1990), no. 6, 1003-1071. 

\bibitem[D-R]{deligne-rapoport}
Deligne, P.; Rapoport, M.
Les sch\'emas de modules de courbes elliptiques. (French) Modular functions of one variable, II (Proc. Internat. Summer School, Univ. Antwerp, Antwerp, 1972), pp. 143-316. Lecture Notes in Math., Vol. 349, Springer, Berlin, 1973. 

\bibitem[Dr]{drinfeld}
Drinfelʹd, V. G. 
Elliptic modules. (Russian) 
Mat. Sb. (N.S.) 94(136) (1974), 594-627, 656. 

\bibitem[H-R]{haines-rapoport}
Haines, T.; Rapoport, M. Shimura varieties with $\Gamma_1(p)$-level via Hecke algebra isomorphisms: the Drinfeld case, Ann. Scient. Ecole Norm. Sup., 4e serie, t. 45 (2012), 719-785.

\bibitem[H-T]{harris-taylor}
Harris, M.; Taylor, R.
The geometry and cohomology of some simple Shimura varieties. 
With an appendix by Vladimir G. Berkovich. Annals of Mathematics Studies, 151. Princeton University Press, Princeton, NJ, 2001. viii+276 pp. ISBN: 0-691-09090-4 

\bibitem[K-M]{katz-mazur}
Katz, N.; Mazur, B. 
Arithmetic moduli of elliptic curves. 
Annals of Mathematics Studies, 108. Princeton University Press, Princeton, NJ, 1985. xiv+514 pp. ISBN: 0-691-08349-5; 0-691-08352-5 

\bibitem[O-S-S]{oss}
Sekiguchi, T; Oort, F.; Suwa, N. 
On the deformation of Artin-Schreier to Kummer. 
Ann. Sci. \'Ecole Norm. Sup. (4) 22 (1989), no. 3, 345-375. 

\bibitem[O-T]{oort-tate}
Tate, J.; Oort, F.
Group schemes of prime order. 
Ann. Sci. \'Ecole Norm. Sup. (4) 3 1970 1-21. 

\bibitem[P]{pappas}
Pappas, G;
Arithmetic models for Hilbert modular varieties. 
Compositio Math. 98 (1995), no. 1, 43-76.

\newcommand{\etalchar}[1]{$^{#1}$}
\bibitem[S{\etalchar{+}}]{sage}
\emph{{S}ageMath, the {S}age {M}athematics {S}oftware {S}ystem}, The Sage Developers, 2015, {\tt http://www.sagemath.org}.



\end{thebibliography}
\end{document}